\documentclass{ams}
\usepackage{bbm}
\usepackage{mathtools}
\usepackage{chemarrow}
\usepackage{amsfonts}
\usepackage{mathrsfs}
\usepackage{stmaryrd}
\usepackage{amssymb}
\usepackage{graphicx}
\usepackage{epstopdf}

\newcommand{\dn}{\mathord{\downarrow}\hspace{0.05em}}
\newcommand{\up}{\mathord{\uparrow}\hspace{0.05em}}
\newtheorem{theorem}{Theorem}[section]
\newtheorem{lemma}[theorem]{Lemma}

\theoremstyle{definition}
\newtheorem{definition}[theorem]{Definition}
\newtheorem{example}[theorem]{Example}

\newtheorem{proposition}[theorem]{Proposition}

\theoremstyle{remark}
\newtheorem{remark}[theorem]{Remark}

\numberwithin{equation}{section}

\begin{document}

\title{Nonexistence of $k$-bounded sobrification}

\author{Jing Lu}
\address{School of Mathematics and Information Science, Shaanxi Normal University, Xi'an 710119, P.R. China}
\email{lujing0926@126.com }
\thanks{Corresponding author: Kaiyun Wang}

\thanks{This work is supported by the National Natural Science Foundation of
China (Grant nos. 11531009, 11871320), Innovation Capability Support Program of Shaanxi (Program No. 2020KJXX-066) and Singapore Ministry of Education
Academic Research Fund Tier 2 grant MOE2016-T2-1-083 (M4020333),
NTU Tier 1 grants RG32/16 (M4011672).}

\author{Kaiyun Wang}
\address{School of Mathematics and Information Science, Shaanxi Normal University, Xi'an 710119, P.R. China}
\email{wangkaiyun@snnu.edu.cn}

\author{Guohua Wu}
\address{Division of Mathematical Sciences, School of Physical and Mathematical
Sciences, Nanyang Technological University,  637371, Singapore}
\email{guohua@ntu.edu.sg}

\author{Bin Zhao}
\address{School of Mathematics and Information Science, Shaanxi Normal University, Xi'an 710119, P.R. China}
\email{zhaobin@snnu.edu.cn}


\date{}

\thanks{}
\begin{abstract}
In this paper, we will focus on $k$-bounded sober spaces and show the existence of a $T_0$ space $X$ not admitting any  $k$-bounded sobrification. This strengthens a result of Zhao, Lu and Wang, who proved that the canonical $k$-bounded sobrification does not exist. Our work provides a complete solution to a question of Zhao and Ho, and shows that unlike  {\bf Sob}
and {\bf BSob},
the category  {\bf KBSob} of all $k$-bounded sober spaces is not a reflective subcategory of the category {\bf Top$_0$} of all $T_0$ spaces.
Furthermore, we introduce the    notion of $qk$-bounded sober spaces and prove that
the category  {\bf KBSob} is a full reflective subcategory of the category {\bf QKBSob} of all $qk$-bounded sober spaces
and continuous mappings preserving existing irreducible suprema.
\end{abstract}

\maketitle

\section{Introduction}

We know that every $T_0$ spaces has a sobrification (see \cite{GG03}), or equivalently,  in terms of category, {\bf Sob}, the
category  of all  sober spaces, is  a full reflective subcategory of the category {\bf Top$_0$} of all $T_0$ spaces.
In \cite{DZ10},  Zhao and Fan introduced bounded sobriety, a weak notion of sobriety, and  proved that
every $T_0$ space has a bounded sobrification. Equivalently,  {\bf BSob}, the
category  of all bounded sober spaces, is  a full reflective subcategory of {\bf Top$_0$}. Being motivated by the definition of the
Scott topology on posets, Zhao and Ho  introduced in \cite{DZ15} a
method of deriving a new topology $\tau_{SI}$ from a given
$\tau$, in a manner similar to the one of deriving  Scott topology from the Alexandroff topology on a
poset. They named this
topology as irreducibly-derived topology, or,
$SI$-topology, for short. It is this $SI$-topology that leads Zhao and Ho to define an even weaker notion of sobriety of $T_0$ spaces in the same paper, the $k$-bounded sobriety.  Zhao and Ho proved that
a $T_0$ space $(X, \tau)$ is  $k$-bounded sober if and only if $\tau=\tau_{SI}$, and that the Scott spaces of continuous posets are always $k$-bounded sober.

After seeing that every $T_0$ space has a sobrification and a bounded sobrification, Zhao and Ho raised a question in \cite{DZ15}, whether $KB(X)$, the set of all irreducible closed sets of a $T_0$ space $X$ whose suprema
exist, is the canonical $k$-bounded sobrification of $X$ in the sense of Keimel and Lawson (see \cite{KK09}). In \cite{ZLW19}, Zhao, Lu and Wang provided a negative answer to it. That is, Zhao, Lu and Wang   added one top element above the dcpo constructed by Johnstone (see \cite{PJ81}), showing that the canonical $k$-bounded sobrification does not exist.
It left open whether any $T_0$ space admits a
$k$-bounded sobrification, which may not be canonical. In this paper, in section 3, we will give a negative answer to it, by showing the existence of a $T_0$ space admitting no $k$-bounded sobrification at all. Our result strengthens the work of Zhao, Lu and Wang above, and provides a complete answer to the question of Zhao and Ho. This shows that unlike  {\bf Sob}
and {\bf BSob},
the category {\bf KBSob} of all $k$-bounded sober spaces is not a reflective subcategory of the category {\bf Top$_0$} of all $T_0$ spaces.
Furthermore,  we show that
the category {\bf KBSob}
is not a  reflective subcategory of the category {\bf Top$_k$}, where {\bf Top$_k$} denotes the category of all  $T_{0}$ spaces and continuous mappings preserving existing irreducible suprema. Finally, we introduce the notion of $qk$-bounded sober spaces and prove that
the category  {\bf KBSob} is a full reflective subcategory of the category {\bf QKBSob} of all $qk$-bounded sober spaces and continuous mappings preserving existing irreducible suprema.

Throughout the paper, we refer to \cite{BA02,GG03} for order
theory,  \cite{RE77} for  general topology and   \cite{JA90} for category theory.

\section{Basics of $k$-bounded sober spaces}
Fix $P$ as a poset. A non-empty subset $D$ of   $P$ is {\em directed}
if every finite subset of $D$ has an upper bound in $D$. A subset
$A$ of $P$ is an {\em upper set} if $A=\up\!A=\{x\in P\mid x\geq  y$ for some
$y\in A\}$. The Alexandroff topology $\Upsilon(P)$ on $P$ is the
topology consisting of all its upper subsets.
A subset $U$ of  $P$ is {\em Scott open} if (i) $U$$=\uparrow\!U$ and (ii) for any directed subset $D$, $\bigvee D\in U$ implies
$D\cap U\neq \emptyset$ whenever $\bigvee D$ exists. Scott open sets on  $P$ form the Scott topology $\sigma(P)$ (see [\cite{GG03}, p. 134]).

 Let $(X, \tau)$ be a topological space. A non-empty subset $F$ of $X$ is  {\em $\tau$-irreducible}, or simply, {\em irreducible},
if whenever $F\subseteq A\cup B$ for closed sets $A$, $B\subseteq X$, one has either $F\subseteq A$ or $F\subseteq B$. A topological space is said to be {\em sober} if every irreducible  closed set  is the closure of a unique singleton.
For a $T_{0}$ space $(X,\tau)$, the {\em specialization order} $\leq$ on $X$ is defined by $x\leq y$ if and only if $x\in cl(\{y\})$ (see [\cite{GG03}, p. 42]). Unless otherwise stated, throughout
the paper, whenever an order-theoretic concept is mentioned in the context of a $T_{0}$ space $X$, it is to be interpreted with respect
to the specialization order on $X$.

\begin{proposition} (\cite{GG03}) Let $X$ be a $T_0$ space.   Then

\begin{enumerate}
\item[(1)] for  $a\in X$, $\dn a := \{x\in X\mid x\leq a\}=cl(\{a\})$.

\item[(2)] if $U\subseteq X$ is an open subset, then  $\up U=U$.

\item[(3)] if $U\subseteq X$ is a closed subset, then $\dn U=U$.

\item[(4)] for any subset $A\subseteq X$,
 $\bigvee\! A$ exists if and only if $\bigvee\! cl(A)$ exists; moreover,  $\bigvee\! A=\bigvee\! cl(A)$, if they exist.
 \end{enumerate}
\end{proposition}

\begin{definition}(\cite{DZ15}) Let $(X,\tau)$ be a $T_{0}$ space. A subset $U$ of $X$ is called $SI$-open if the following conditions are satisfied:
\begin{enumerate}
\item[(1)] $U\in\tau$;

\item[(2)] for every  irreducible subset $F$ of $X$, $\bigvee F\in U$ implies $F\cap U\neq\emptyset$ whenever $\bigvee F$ exists.
 \end{enumerate}
\end{definition}

 The set of all $SI$-open subsets of $(X,\tau)$ is denoted by
$\tau_{SI}$. It is straightforward to verify that $\tau_{SI}$
is a topology on $X$. We call $\tau_{SI}$  the irreducibly-derived
topology (or simply, $SI$-topology).\ The space $(X,\tau_{SI})$ will
also be simply written as $SI(X)$. Complements of $SI$-open sets are
called $SI$-closed sets. Note that $SI(P,\Upsilon(P))=(P,\sigma(P))$ for every poset $P$.

\begin{proposition}(\cite{DZ15}) A closed subset $C$ of a $T_{0}$ space $(X,\tau)$ is $SI$-closed if and only if for every  irreducible subset $F$ of $X$, $F\subseteq C$ implies $\bigvee F\in C$ whenever
$\bigvee F$ exists.
\end{proposition}

\begin{proposition}(\cite{DZ15}) A continuous mapping $f : (X, \tau)\longrightarrow (Y, \delta)$ is a  continuous mapping between $SI(X)$ and $SI(Y)$ if and only if  $f(\bigvee F)=\bigvee f(F)$ holds for each irreducible subset  $F$ in $X$ whenever $\bigvee F$ exists.
\end{proposition}

Let $X$ and $Y$ be $T_{0}$ spaces. A mapping $f: X\longrightarrow Y$  is called   preserving existing irreducible suprema if $f(\bigvee F)=\bigvee f(F)$ holds for each irreducible subset  $F$ in $X$ whenever $\bigvee F$ exists.

\begin{definition}(\cite{DZ15}) A space $(X, \tau)$ is $k$-bounded sober if for any irreducible closed subset $F$ of $X$
whose $\bigvee F$ exists, there is a unique point $x\in X$ such that
$F=cl(\{x\})$.
\end{definition}

\begin{proposition}(\cite{DZ15}) A $T_{0}$ space $(X, \tau)$ is $k$-bounded sober if and only if $\tau=\tau_{SI}$.
\end{proposition}

A $T_0$ space is said to be {\em bounded sober} (see \cite{DZ10}) if every upper bounded  irreducible closed set is the closure of a unique singleton. In \cite{DZ15}, Zhao and Ho presented that the following implications hold:
\begin{center}
\centering
sobriety$\Longrightarrow$ bounded sobriety$\Longrightarrow$ $k$-bounded sobriety.
\end{center}

Let {\bf KBSob} denote the category of all $k$-bounded sober spaces and continuous mappings.

It is well-known that any product of sober spaces is sober (see \cite{GG03}), where in the product space, irreducible closed sets are products of irreducible closed subsets from factor spaces.

\begin{lemma} (\cite{JG13}) Let $\{X_i\}_{i\in I}$ be a family of topological spaces. In the product space $\prod_{i\in I}X_i$, irreducible closed sets are subsets of $\prod_{i\in I}X_i$  in
the form $\prod_{i\in I}C_i$, where for each $i\in I$, $C_i$ is an irreducible closed subset  of $X_i$.
\end{lemma}

The following proposition says that $k$-bounded sober spaces
 are also closed under products.

\begin{proposition} Products of $k$-bounded sober spaces are $k$-bounded sober.
\end{proposition}

\begin{proof}Let $\{X_i\}_{i\in I}$ be a family of $k$-bounded sober spaces and $C$  an irreducible
closed subset of  $\prod_{i\in I}X_i$  whose  $\bigvee C$ exists. By Lemma 2.7, $C=\prod_{i\in I}C_i$, where for each $i\in I$, $C_i$ is an irreducible closed subset  of $X_i$.

Let $\bigvee C=(c_i)_{i\in I}$. Recall that  the specialization order $\leq$ on $\prod_{i\in I}X_i$ is the product of the specialization orders $\leq_i$. That is, for  $x=(x_i)_{i\in I}, y=(y_i)_{i\in I}\in \prod_{i\in I}X_i$, $x\leq y$ if and only if $x_i\leq_i y_i$ for all $i\in I$, where $\leq_i$ is the specialization order on $X_i$. This implies that for each $i\in I$, $\bigvee C_i=c_i$. Since $X_i$ is a $k$-bounded sober space, we have that $C_i=\dn c_i$, and hence, \[C=\prod_{i\in I}C_i=\prod_{i\in I}\dn c_i=\dn(c_i)_{i\in I},\] showing that  $\prod_{i\in I}X_i$ is $k$-bounded sober.  \end{proof}

\section{A $T_0$ space admitting no $k$-bounded sobrification}

In this section, we will construct a $T_0$ space admitting no $k$-bounded sobrification at all.

\begin{definition}A  $k$-bounded sobrification of a $T_{0}$ space $(X,\tau)$ consists of a k-bounded sober space $\overline {X}$ and a continuous mapping $\xi:X\longrightarrow \overline {X}$
with the following universal property: For every continuous
mapping $f$ from $X$ to a k-bounded sober space $S$, there exists a
unique continuous mapping $\bar{f}: \overline {X}\longrightarrow
S$ such that $\bar{f}\circ \xi=f$, i.e., the following
diagram commutes:\\
\newline
\begin{picture}(150,100)(50,50)
\put(275,81){\makebox(0,0){$S$}}
\put(275,135){\makebox(0,0){$\overline {X}$}}
\put(220,135){\vector(1,0){45.0}}
\put(215,135){\makebox(0,0){$X$}}
\put(250,142){\makebox(0,0){$\xi$}}
\put(230,100){\makebox(0,0){$f$}}
\put(283,110){\makebox(0,0){$\overline {f}$}}
 \put(220,130){\vector(1,-1){45.0}}
\put(275,125){\vector(0,-1){35.0}}
\put(270,60){\makebox(0,0){Figure 1}}
\end{picture}
\end{definition}

\begin{remark} We note that for a $T_{0}$ space $X$, if $(\overline {X},\xi)$ is the  $k$-bounded sobrification of $X$, then $\xi$ is actually an embedding mapping.
Indeed, $\xi$ is a continuous mapping. Define a set $X^{S}$ by
$$X^{S}=\{A\subseteq X\mid A \mbox{ is an irreducible closed set}\}.$$
Topologize $X^{S}$ by open sets $U^{S}=\{A\in X^{S}\mid A\cap U\neq \emptyset\}$
for each open set $U$ of $X$.
Then $(X^{S}, \dn_{X})$ is the sobrification of $X$ (see, e.g., \cite{GG03}, Exercise V-4.9). Here, $\dn_{X}: X\longrightarrow X^{S}$ is defined by $\dn_{X}(x)=\dn x$ for  $x\in X$.
Hence, $X^{S}$ is a k-bounded sober space. By Definition 3,1, there exists  a unique continuous mapping $\overline {\dn_{X}}: \overline {X}\longrightarrow
X^{S}$ such that $\overline {\dn_{X}}\circ \xi=\dn_{X}$, i.e., the following
diagram commutes:\\
\newline
\begin{picture}(150,100)(50,50)
\put(275,81){\makebox(0,0){$X^S$}}
\put(275,135){\makebox(0,0){$\overline {X}$}}
\put(220,135){\vector(1,0){45.0}}
\put(215,135){\makebox(0,0){$X$}}
\put(250,142){\makebox(0,0){$\xi$}}
\put(232,100){\makebox(0,0){$\dn_{X}$}}
\put(286,110){\makebox(0,0){$\overline {\dn_{X}}$}}
 \put(220,130){\vector(1,-1){45.0}}
\put(275,125){\vector(0,-1){35.0}}
\put(270,60){\makebox(0,0){Figure 2}}
\end{picture}\\
Since $\dn_{X}$ is  injective,  $\xi$ is also injective.
Let $U$ be an open set of $X$. Then $\xi(U)=\xi(X)\cap\overline {\dn_{X}}^{-1}(U^{S})$, and thus $\xi(U)$ is an open subset of the
subspace $\xi(X)$. This shows that $\xi$ is an embedding mapping.
\end{remark}

We now give   an example showing that $k$-bounded
sobrifications of a $T_0$ space $X$ may not exist. Our construction is motivated by Johnstone's famous
counterexample (see \cite{PJ81}).

\begin{theorem} \label{nonsobrification} There exists a $T_0$ space $X$ admitting no  $k$-bounded
sobrification.
\end{theorem}

\begin{proof} Let $P=\big(\mathbb{N}\times (\mathbb{N}\cup\{\infty\})\big)\cup\{\top\}$. Define a partial order on $P$ as follows: for $x, y\in P$, $x\leq y$ iff one of the following
conditions holds (see Figure 3):
\begin{enumerate}
\item[(1)] $y=\top$;

\item[(2)] $x=(m_{1},n_{1})$ and $y=(m_{2},n_{2})$, with $m_{1}=m_{2}$, $n_{1}\leq n_{2}\ (\leq \infty)$ or $n_{2}= \infty$, $n_{1}\leq
m_{2}$.
\end{enumerate}

\begin{center}
\centering
\includegraphics[totalheight=1.8in]{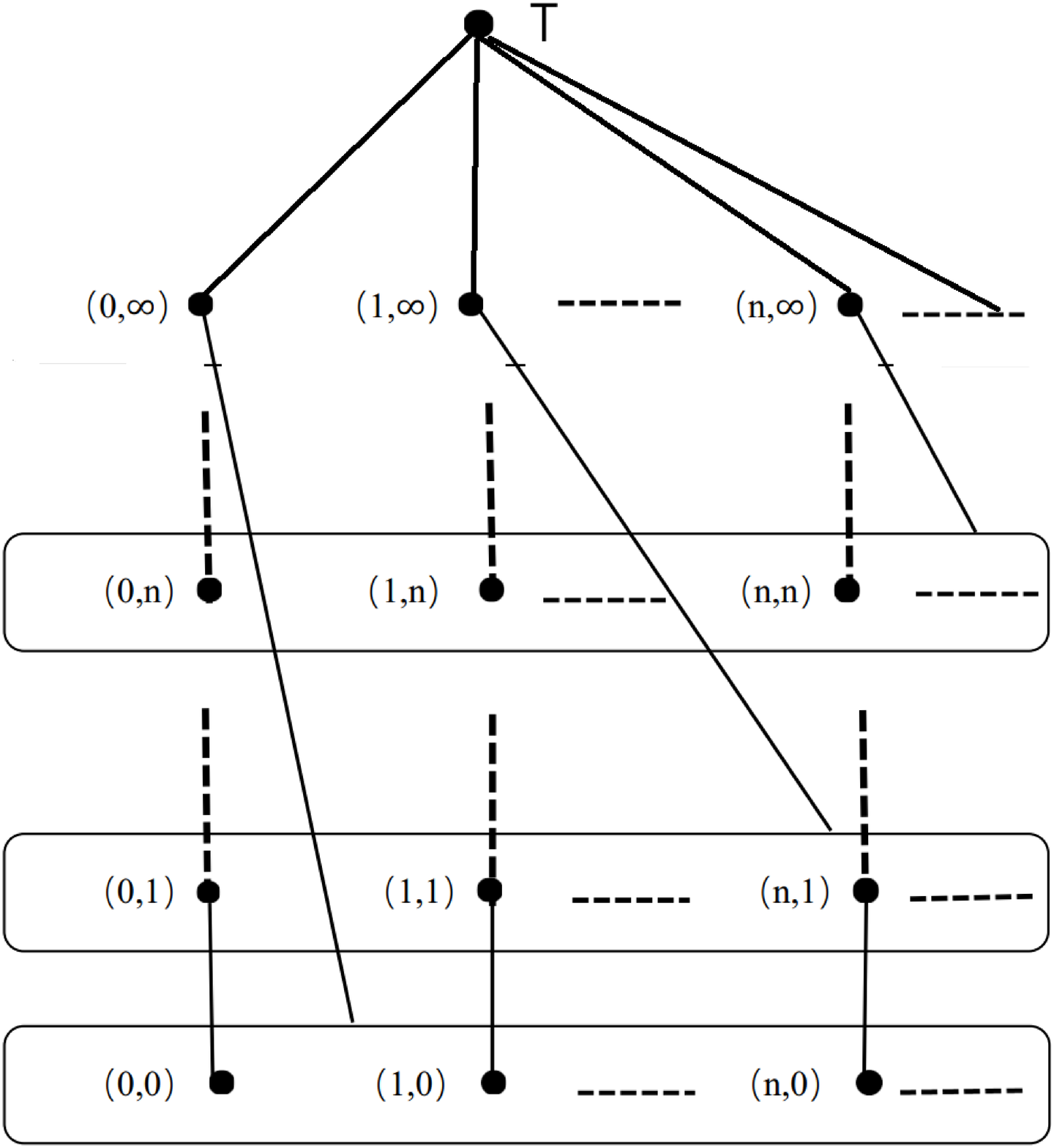}\\Figure 3
\end{center}

That is, we  add $\top$ above all elements in Johnstone's dcpo.   We will show below that the  $k$-bounded sobrification
of the $T_{0}$ space $(P, \sigma(P))$ does not exist.

 Assume for contradiction that the   $k$-bounded sobrification
of   $(P, \sigma(P))$  exists. Let it be $(\overline {P}, i)$ and let $A=\mathbb{N}\times (\mathbb{N}\cup\{\infty\})$.
Then $A$ is an irreducible Scott closed  subset of $P$ with $\bigvee A=\top$. Since $i$ is  continuous, $i(A)$ is an irreducible subset of $\overline {P}$. There are two cases:

\vspace*{0.2cm}
{\bf Case 1}. $\bigvee i(A)$ exists in $\overline {P}$.
\vspace*{0.2cm}

In this case, we introduce two new maximal elements $\top_1$ and $\top_2$, and let $X=\big(\mathbb{N}\times
(\mathbb{N}\cup\{\infty\})\big)\cup\{\top_{1}, \top_{2}\}$. Define a partial order $\leq$ on $X$ as follows (see Figure 4):

\begin{enumerate}
\item[(1)] for  $x\in \mathbb{N}\times (\mathbb{N}\cup\{\infty\})$, $x\leq\top_{1}, \top_{2}$;

\item[(2)] for $x, y\in \mathbb{N}\times (\mathbb{N}\cup\{\infty\})$, if
$x=(m_{1},n_{1})$, $y=(m_{2},n_{2})$, $x\leq y$ iff $m_{1}=m_{2}$, $n_{1}\leq n_{2}\ (\leq \infty)$ or $n_{2}= \infty$, $n_{1}\leq m_{2}$.
\end{enumerate}
\begin{center}
\centering
\includegraphics[totalheight=2.2in]{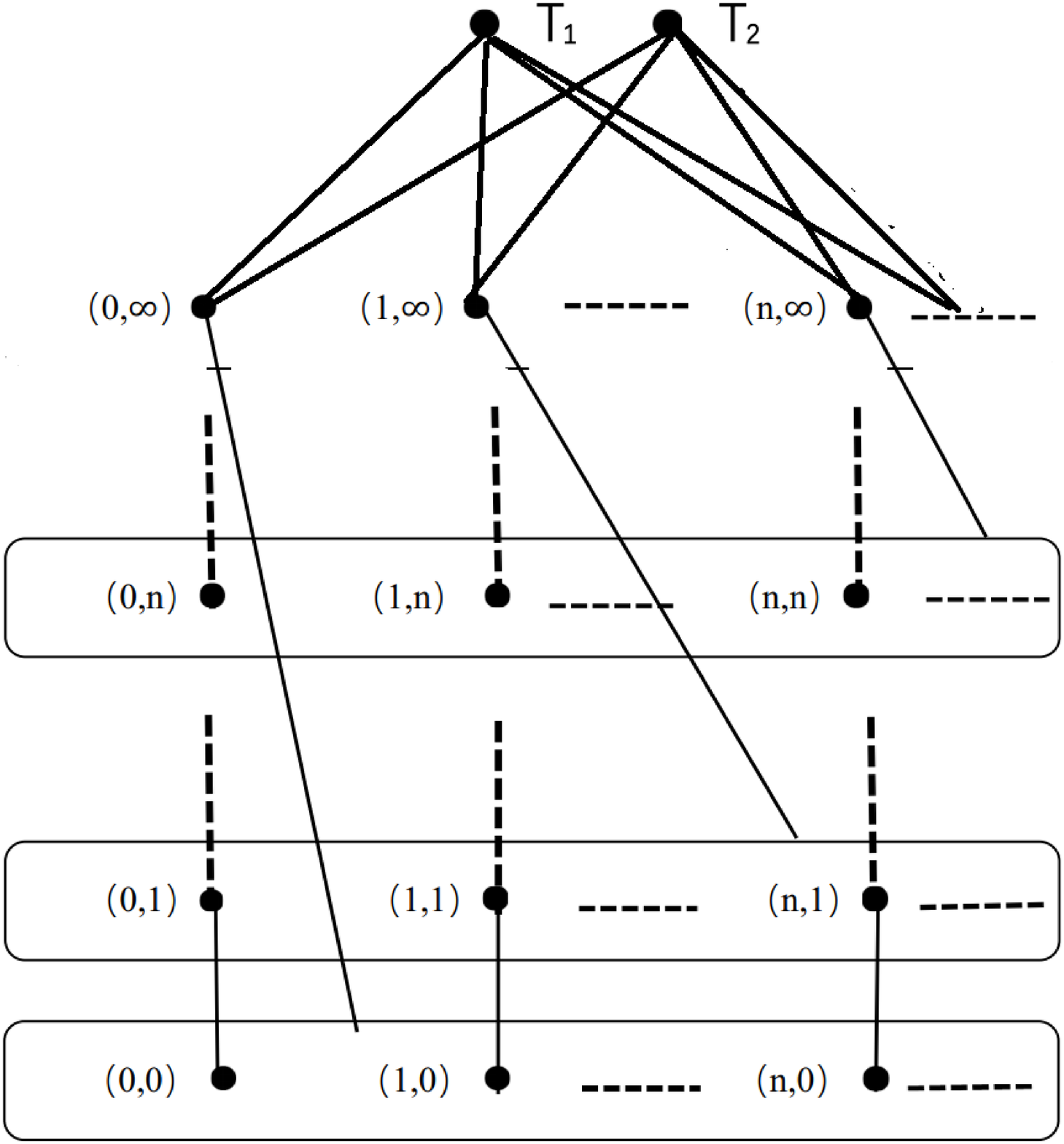}\\Figure 4
\end{center}
It is easy to check that $(X,\sigma(X))$ is a $k$-bounded sober space. Define a mapping $f: P\longrightarrow X$ as follows:
\begin{center}
$\forall \ a\in P,\  f(a)=\left\{\begin{array}{ll}
a,& \ \ \ \ \hbox{if}\ a\in A,\\
\top_{1}, &\ \ \ \ \hbox{if}\ a=\top.\end{array}\right. $
\end{center}
Then $f$ is  Scott continuous and $\bigvee f(A)$ does not exist in $X$. Since  $(\overline {P}, i)$ is assumed to be the $k$-bounded sobrification
of  $(P, \sigma(P))$, there exists a unique continuous mapping $\bar{f} : \overline {P}\longrightarrow X$ such that $\bar{f}\circ i=f$, i.e., the following
diagram commutes:\\
\newline
\begin{picture}(150,100)(50,50)
\put(275,81){\makebox(0,0){$X$}}
\put(275,135){\makebox(0,0){$\overline {P}$}}
\put(220,135){\vector(1,0){45.0}}
\put(215,135){\makebox(0,0){$P$}}
\put(250,142){\makebox(0,0){$i$}}
\put(230,100){\makebox(0,0){$f$}}
\put(283,110){\makebox(0,0){$\overline {f}$}}
\put(220,130){\vector(1,-1){45.0}}
\put(275,125){\vector(0,-1){35.0}}
\put(265,60){\makebox(0,0){Figure 5}}
\end{picture}\\
Since $\overline {P}$ and $(X,\sigma(X))$ are both $k$-bounded sober, by Proposition 2.4, we have that $$\bar{f}\Big(\bigvee i(A)\Big)=\bigvee \bar{f}(i(A))=\bigvee f(A).$$
This contradicts the fact that $\bigvee f(A)$ does not exist in $X$.

\vspace*{0.2cm}
{\bf Case 2}.   $\bigvee i(A)$ does not  exist in $\overline {P}$.
\vspace*{0.2cm}

Note that  $i(\top)$ is an upper bound of  $i(A)$. In this case, there is an upper bound  $x_0$ of $i(A)$ in $\overline {P}$ such that $i(\top)\nleqslant x_0$.  We introduce three new elements
$\top_{1}, \top_{2}, \top_{3}$ and
let $Y=\big(\mathbb{N}\times
(\mathbb{N}\cup\{\infty\})\big)\cup\{\top_{1}, \top_{2}, \top_{3}\}$. Define a partial order
$\leq$ on $Y$ as follows (see Figure 6):
\begin{enumerate}
\item[(1)] $\top_{1}\leq \top_{3}, \top_{2}\leq \top_{3}$;

\item[(2)] for   $x\in \mathbb{N}\times (\mathbb{N}\cup\{\infty\}), x\leq\top_{1}, \top_{2}$;

\item[(3)]  for $x, y\in \mathbb{N}\times (\mathbb{N}\cup\{\infty\})$, if
$x=(m_{1},n_{1})$, $y=(m_{2},n_{2})$, $x\leq y$ iff $m_{1}=m_{2}$, $n_{1}\leq n_{2}\ (\leq \infty)$ or $n_{2}= \infty$, $n_{1}\leq m_{2}$.
\end{enumerate}
\begin{center}
\centering
\includegraphics[totalheight=2.2in]{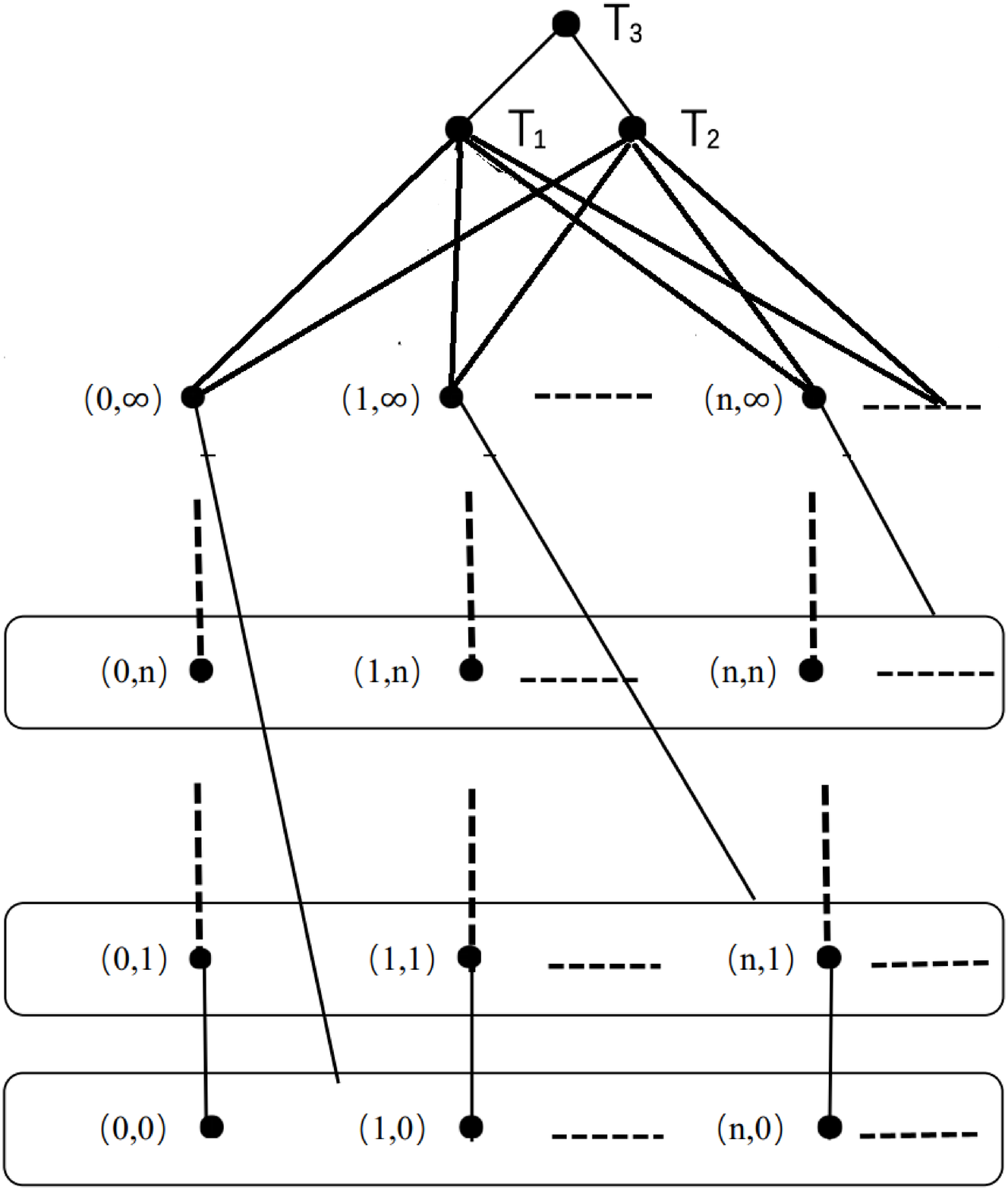}\\Figure 6
\end{center}
It is easy to check that $(Y,\sigma(Y))$ is a $k$-bounded sober space. Define a mapping $f: P\longrightarrow Y$ as follows:
\begin{center}
$\forall \ a\in P,\  f(a)=\left\{\begin{array}{ll}
a,& \ \ \ \ \hbox{if}\ a\in A,\\
\top_{3}, &\ \ \ \ \hbox{if}\ a=\top.\end{array}\right. $
\end{center}
Again, $f$ is  Scott continuous. Since  $(\overline {P}, i)$ is the  $k$-bounded sobrification
of  $(P, \sigma(P))$, there exists a unique continuous mapping $\bar{f} : \overline {P}\longrightarrow Y$ such that $\bar{f}\circ i=f$, i.e., the following
diagram commutes:\\
\newline
\begin{picture}(150,100)(50,50)
\put(275,81){\makebox(0,0){$Y$}}
\put(275,135){\makebox(0,0){$\overline {P}$}}
\put(220,135){\vector(1,0){45.0}}
\put(215,135){\makebox(0,0){$P$}}
\put(250,142){\makebox(0,0){$i$}}
\put(230,100){\makebox(0,0){$f$}}
\put(283,110){\makebox(0,0){$\overline {f}$}}
\put(220,130){\vector(1,-1){45.0}}
\put(275,125){\vector(0,-1){35.0}}
\put(265,60){\makebox(0,0){Figure 7}}
\end{picture}

\vspace*{0.2cm}
{\bf Subcase 1}. $\bar{f}^{-1}(\top_{1})\cup \bar{f}^{-1}(\top_{2})\neq \emptyset$.
\vspace*{0.2cm}

In this case, there is $x_1\in \overline {P}$ such that $\bar{f}(x_1)=\top_{1}$ or  $\bar{f}(x_1)=\top_{2}$. Define a mapping $g : Y\longrightarrow Y$ as follows:
\begin{center}
$ \forall \ a\in Y,\ g(a)=\left\{\begin{array}{ll}
a,& \ \ \ \ \hbox{if}\ a\in A,\\
\top_{3}, &\ \ \ \ \hbox{otherwise}.\end{array}\right. $
\end{center}

Then $g$ is  Scott continuous with $g\circ f=f$. Thus
$(g\circ \bar{f})\circ i=g\circ (\bar{f}\circ i)=g\circ f=f,$ i.e., the following
diagram commutes:\\
\newline
\begin{picture}(150,100)(50,50)
\put(275,81){\makebox(0,0){$Y$}}
\put(275,135){\makebox(0,0){$\overline {P}$}}
\put(220,135){\vector(1,0){45.0}}
\put(215,135){\makebox(0,0){$P$}}
\put(250,142){\makebox(0,0){$i$}}
\put(230,100){\makebox(0,0){$f$}}
\put(263,110){\makebox(0,0){$\overline {f}$}}
\put(295,110){\makebox(0,0){$g\circ\overline {f}$}}
\put(220,130){\vector(1,-1){45.0}}
\put(270,125){\vector(0,-1){35.0}}
\put(278,125){\vector(0,-1){35.0}}
\put(265,60){\makebox(0,0){Figure 8}}
\end{picture}
\newline By the uniqueness of $\bar{f}$, $g\circ \bar{f}=\bar{f}$, contradicting the fact that $(g\circ \bar{f})(x_1)=\top_{3}\neq \bar{f}(x_1)$.

\vspace*{0.2cm}
{\bf Subcase 2}.  $\bar{f}^{-1}(\top_{1})\cup \bar{f}^{-1}(\top_{2})=\emptyset$.
\vspace*{0.2cm}

Since $i(\top)\in \overline {P}\backslash \bar{f}^{-1}(A)$,\ $\overline {P}\backslash \bar{f}^{-1}(A)\neq \emptyset$.    In this case, for $x\in \overline {P}\backslash \bar{f}^{-1}(A)$, $\bar{f}(x)=\top_{3}$. Clearly, $\bar{f}(x_0)=\top_3$.
Define a mapping $\hat{f}: \overline {P}\longrightarrow Y$ as follows:
\begin{center}
$\forall \ x\in  \overline {P},\  \hat{f}(x)=\left\{\begin{array}{cl}
\bar{f}(x), & \ \ \ \ \hbox{if}\ x\in \bar{f}^{-1}(A);\\
\top_{1},& \ \ \ \ \hbox{if}\ x\in \bar{f}^{-1}(\up \top_3)\cap \dn x_0;\\
\top_{3},& \ \ \ \ \hbox{if}\ x\in \bar{f}^{-1}(\up \top_3)\cap (\overline {P}\setminus \dn x_0).
\end{array}\right. $
\end{center}
We claim that $\hat{f}$ is continuous. In fact, let $B$ be a Scott closed subset in $Y$. If $B\subseteq A$, then $\hat{f}^{-1}(B)=\bar{f}^{-1}(B)$. Since $\bar{f}$ is a continuous mapping, $\hat{f}^{-1}(B)$ is a closed subset in $\overline {P}$. If $B\nsubseteq A$, then $B=\dn \top_1$ or $B=\dn \top_2$ or $B=\dn \top_1\cup \dn \top_2$ or $B=Y$. Obviously, $A$ is a Scott closed subset in $Y$. Since  $\bar{f}^{-1}(\top_{1})\cup \bar{f}^{-1}(\top_{2})=\emptyset$, we have that $\bar{f}^{-1}(\up \top_3)\cup \bar{f}^{-1}(A)=\overline {P}$. Then
$$\hat{f}^{-1}(\dn \top_1)=(\bar{f}^{-1}(\up \top_3)\cap \dn x_0)\cup \bar{f}^{-1}(A)=\overline {P}\cap (\dn x_0\cup \bar{f}^{-1}(A))=   \dn x_0\cup \bar{f}^{-1}(A).$$
So $\hat{f}^{-1}(\dn \top_1)$ is a closed subset in $\overline {P}$. As $\hat{f}^{-1}(\dn \top_2)=\bar{f}^{-1}(A)$, we conclude that $\hat{f}^{-1}(\dn \top_2)$ is a closed subset in $\overline {P}$. Then $\hat{f}^{-1}(\dn\top_1\cup \dn \top_2)$ and $\hat{f}^{-1}(Y)$ are closed subsets in $\overline {P}$. Thus $\hat{f}$ is a continuous mapping.
Furthermore, $\hat{f}\circ i=f$, i.e., the following
diagram commutes:\\
\newline
\begin{picture}(150,100)(50,50)
\put(275,81){\makebox(0,0){$Y$}}
\put(275,135){\makebox(0,0){$\overline {P}$}}
\put(220,135){\vector(1,0){45.0}}
\put(215,135){\makebox(0,0){$P$}}
\put(250,142){\makebox(0,0){$i$}}
\put(230,100){\makebox(0,0){$f$}}
\put(263,110){\makebox(0,0){$\overline {f}$}}
\put(285,110){\makebox(0,0){$\hat{f}$}}
 \put(220,130){\vector(1,-1){45.0}}
\put(270,125){\vector(0,-1){35.0}}
\put(278,125){\vector(0,-1){35.0}}
\put(265,60){\makebox(0,0){Figure 9}}
\end{picture}
\newline By the uniqueness of $\bar{f}$,  $\hat{f}=\bar{f}$, contradicting the fact that $\hat{f}(x_0)=\top_1$ and $\bar{f}(x_0)=\top_3$.

This shows that the $T_0$ space $(P, \sigma(P))$ can not have  $k$-bounded sobrification.\end{proof}

By Theorem \ref{nonsobrification}, we have the following theorem immediately.

\begin{theorem}  The category {\bf KBSob} is not a reflective subcategory of {\bf
Top$_{0}$}.
\end{theorem}

Let $KB(X)$ denote the set of all irreducible closed sets of
a $T_{0}$ space $X$ whose suprema exist. We write $K_{F}=\{A\in
KB(X)\mid A\subseteq F\}$ for any closed set $F$ of $X$. It is then
straightforward to verify that $\{K_{F}\mid F\in\Gamma(X)\}$ actually
defines a co-topology on $KB(X)$. Here, $\Gamma(X)$ is the family of
all closed sets of $X$. We will use $KB(X)$ to denote the
corresponding topological space. Zhao and Ho proved in \cite{DZ15} that
$KB(X)$ is a $k$-bounded sober space.

\begin{proposition}  Let $f$ be a continuous mapping from a $T_{0}$ space $X$  to a k-bounded sober space $Y$. If $f$  preserves existing irreducible suprema, then there exists a
unique continuous mapping $\bar{f} : KB(X)\longrightarrow
Y$ such that $\bar{f}\circ \xi=f$, where the mapping $\xi : X\longrightarrow KB(X)$ is defined by $\xi(x)=\dn x$.
\end{proposition}

\begin{proof}For all $F\in KB(X)$, we have that  $f(\bigvee F)=\bigvee f(F)$. Define a mapping $\bar{f} : KB(X)\longrightarrow
Y$ by
$  \bar{f}(F)=f(\bigvee F)$ for $F\in KB(X)$.
Clearly, $\bar{f}\circ \xi=f$. Let $B$ be a closed subset of $Y$. Since $Y$ is a  $k$-bounded sober space,
\[A\in \bar{f}^{-1}(B)\Longleftrightarrow \bar{f}(A)\in B\Longleftrightarrow f(\bigvee A)\in B \Longleftrightarrow \bigvee f(A)\in B\]
\[\Longleftrightarrow f(A)\subseteq B \Longleftrightarrow A\subseteq f^{-1}(B) \Longleftrightarrow A\in K_{f^{-1}(B)}.\]
Thus $\bar{f}^{-1}(B)$ is a closed subset of $KB(X)$, and   $\bar{f}$  is continuous. Let $g : KB(X)\longrightarrow
Y$ be a continuous mapping such that $g\circ \xi=f$.
By Propositions 2.4 and 2.6, we have that $g(\bigvee \xi(F))=\bigvee g(\xi(F))$ for all $F\in KB(X)$, and hence
$$g(F)=g\Big(\bigvee\limits_{x\in F}\dn x\Big)=g(\bigvee \xi(F))=\bigvee g(\xi(F))=\bigvee f(F)=f(\bigvee F)=\bar{f}(F),$$
and $g=\bar{f}$. The proof is completed.\end{proof}

\begin{example} Let $P=\mathbb{N}\cup\{\infty\}$, where $\mathbb{N}$ denotes the set of natural numbers, and let the partial order $\leq$
 on $X$ be the standard order:
 $$0\leq 1\leq\cdot\cdot\cdot\leq n\leq\cdot\cdot\cdot\leq\infty.$$
Obviously,  the Alexandroff topological space $(P,\Upsilon(P))$ is a $T_{0}$ space, but not a  $k$-bounded sober spaces.

Let $\mathbb{N}^{S}=\{A\subseteq \mathbb{N}\mid  A$ \mbox{ is a directed lower set}\}. We consider the topology
 $\mathcal{T}=\{U^{S}\mid U\subseteq\mathbb{N}$ \mbox{ is an upper set}\} on $\mathbb{N}^{S}$, where $U^{S}=\{A\in \mathbb{N}^{S}\mid A\cap U\neq\emptyset\}$.
 Then $(\mathbb{N}^{S},\mathcal{T})$ is the sobrification of the Alexandroff topological space $(\mathbb{N},\Upsilon(\mathbb{N}))$, and hence $k$-bounded sober.
 Define a mapping $\xi: P\longrightarrow \mathbb{N}^{S}$ as follows:
\begin{center}
$\forall \ x\in P,\ \xi(x)=\left\{\begin{array}{ll}
 \dn x, \ &\ \ \ \ \hbox{if}\ x\in \mathbb{N},\\
 \mathbb{N},  & \ \ \ \ \hbox{if}\ x=\infty.\end{array}\right. $
\end{center}
Then $\xi$ is a continuous mapping  preserving existing irreducible suprema.
Since $\xi(\{\infty\})=\{\mathbb{N}\}$ is not an open subset in the subspace $\xi(P)$, we have that $\xi$
is not an embedding.

Let $Y$ be a k-bounded sober space and $f: P\longrightarrow Y$  a continuous mapping preserving
existing irreducible suprema.  Then there exists a
unique continuous mapping $\bar{f} : \mathbb{N}^{S}\longrightarrow
Y$ such that $\bar{f}\circ \xi=f$, i.e., the following
diagram commutes:\\
\newline
\begin{picture}(150,100)(50,50)
\put(275,81){\makebox(0,0){$Y$}}
\put(275,135){\makebox(0,0){$\mathbb{N}^{S}$}}
\put(220,135){\vector(1,0){45.0}}
\put(215,135){\makebox(0,0){$P$}}
\put(250,142){\makebox(0,0){$\xi$}}
\put(230,100){\makebox(0,0){$f$}}
\put(283,110){\makebox(0,0){$\overline {f}$}}
 \put(220,130){\vector(1,-1){45.0}}
\put(275,125){\vector(0,-1){35.0}}
\put(265,60){\makebox(0,0){Figure 10}}
\end{picture}\\
We define a mapping $\bar{f}: \mathbb{N}^{S}\longrightarrow Y$ as follows:
 $$\forall \ A\in \mathbb{N}^{S},\ \bar{f}(A)=\bigvee f(A). $$
 Obviously, $\bar{f}$ is well-defined. Let $V$ be an open subset of $Y$.  Then $f^{-1}(V)$ is an open subset of $P$.
 Whenever $f^{-1}(V)=\emptyset$,  we will prove that
 $\bar{f}^{-1}(V)=\emptyset$.
 Assume that
 $\bar{f}^{-1}(V)\neq\emptyset$. Then there exists $A\in \mathbb{N}^{S}$ such that $A\in\bar{f}^{-1}(V)$, and thus
 $\bar{f}(A)=\bigvee f(A)\in V$. Since $\bigvee A$ exists in $P$, we have that $f(\bigvee A)=\bigvee f(A)\in V$. This means that $\bigvee A\in f^{-1}(V)$,  contradicting the fact that $f^{-1}(V)=\emptyset$. So
 we conclude that $\bar{f}^{-1}(V)=\emptyset$.
 Whenever $f^{-1}(V)\neq\emptyset$, then there exists $a\in P$ such that $f(a)\in V$, and thus $$f(a)\leq f(\infty)=f(\bigvee\mathbb{N})=\bigvee f(\mathbb{N})\in V.$$
 Since $Y$ is a k-bounded sober space, we have that  $\bigvee f(\mathbb{N})\in cl(f(\mathbb{N}))$. Then $f(\mathbb{N})\cap V\neq \emptyset$, and thus
 $f^{-1}(V)\cap \mathbb{N}\neq\emptyset$.
 Since $$A\in\bar{f}^{-1}(V)\Longleftrightarrow \bigvee f(A)\in V\Longleftrightarrow f(A)\cap V\neq\emptyset\Longleftrightarrow A\cap f^{-1}(V)\cap\mathbb{N}\neq\emptyset,$$
 $\bar{f}^{-1}(V)=(f^{-1}(V)\cap\mathbb{N})^{S}$. This shows that  $\bar{f}$ is continuous.
 By $\bar{f}(\xi(x))=f(x)$ for any $x\in P$, we have that $\bar{f}\circ \xi=f$,
\newline Suppose that there exists a continuous mapping $g: \mathbb{N}^{S}\longrightarrow Y$ such that $g\circ \xi=f$.
 Then $g(\dn x)=f(x)$ for $x\in \mathbb{N}$. Let $A\in \mathbb{N}^{S}$. Then $$g(A)=g\Big(\bigvee\limits_{a\in A}\dn a\Big)=\bigvee\limits_{a\in A}g(\dn a)=\bigvee\limits_{a\in A}f(a)=\bar{f},$$ and thus $g=\bar{f}$.
\end{example}

\begin{remark} (1) By the  example above,  if the  mappings $\xi$ and $f$ in Definition 3.1 preserve  existing irreducible suprema, then $\xi$ may not be an embedding mapping. Note that in Remark 3.2, the corresponding mappings are just continuous.

(2) If we add the condition that the mappings $\xi$ and $f$ in Definition 3.1 preserve  existing irreducible suprema,
then it follows from Example 3.6 that the Alexandroff topological space $(P,\Upsilon(P))$  has a  $k$-bounded sobrification in this sense.
\end{remark}

Let {\bf Top$_k$} denote the category of all $T_{0}$ spaces and continuous mappings preserving existing irreducible suprema.
By Propositions 2.4 and 2.6, we have that {\bf KBSob} is a full subcategory of {\bf Top$_k$}.

\begin{proposition}
The category {\bf KBSob}
is  not a reflective subcategory of {\bf Top$_k$}.
\end{proposition}

\begin{proof} Assume that  {\bf KBSob} is a reflective subcategory of {\bf Top$_k$}. Let $P$ be the poset as in Theorem 3.3. We consider the Alexandroff topological space $(P,\Upsilon(P))$. Then there exist a $k$-bounded sober space $\overline {P}$ and a continuous mapping preserving existing irreducible suprema
$\eta: (P,\Upsilon(P))\longrightarrow \overline {P}$ which enjoy the universal property: for any continuous mapping preserving existing irreducible suprema  $h$ from $(P,\Upsilon(P))$ to a
$k$-bounded sober space $Z$, there exists a unique continuous mapping $\overline{h}: \overline {P}\longrightarrow Z$ such that $\overline{h}\circ\eta=h$.
By Propositions 2.4 and 2.6,
$\eta: (P,\sigma(P))\longrightarrow \overline {P}$ is a continuous mapping.
Let $A=\mathbb{N}\times (\mathbb{N}\cup\{\infty\})$. Then $A$ is an irreducible Scott closed  subset of $P$ with $\bigvee A=\top$, and thus  $\eta(A)$ is an irreducible subset of $\overline {P}$. There are two cases:

\vspace*{0.2cm}
{\bf Case 1}. $\bigvee \eta(A)$ exists in $\overline {P}$.
\vspace*{0.2cm}

Let $X$ be the poset as in Case 1 of the proof of Theorem 3.3. Then $(X, \sigma(X))$ is a $k$-bounded sober space. Let $f : P\longrightarrow X$ be the mapping as in Case 1 of the proof of Theorem 3.3. Then $f$ is Scott continuous and $\bigvee f(A)$ does not exist in $X$.  Since $\sigma(P)\subseteq \Upsilon(P)$, $f$ is a continuous mapping from $(P,\Upsilon(P))$ to the
$k$-bounded sober space $(X, \sigma(X))$. By Proposition 2.4, we have that $f(\bigvee F)=\bigvee f(F)$ holds for each irreducible subset  $F$ in $(P,\Upsilon(P))$ whenever $\bigvee F$ exists.  By the universal property, there exists a unique continuous mapping $\overline{f}: \overline {P}\longrightarrow (X, \sigma(X))$ such that $\overline{f}\circ\eta=f$. So $$\bar{f}\Big(\bigvee \eta(A)\Big)=\bigvee \bar{f}(\eta(A))=\bigvee f(A).$$ This contradicts the fact that $\bigvee f(A)$ does not exist in $X$.

\vspace*{0.2cm}
{\bf Case 2}. $\bigvee \eta(A)$ does not exist in $\overline {P}$.
\vspace*{0.2cm}

Let $Y$ be the poset as in Case 2 of the proof of Theorem 3.3. Then $(Y, \sigma(Y))$ is a $k$-bounded sober space. Let $f : P\longrightarrow Y$ be the mapping as in Case 2 of the proof of Theorem 3.3. Then $f$ is Scott continuous.  Since $\sigma(P)\subseteq \Upsilon(P)$,  $f$ is a continuous mapping from $(P,\Upsilon(P))$ to $(Y, \sigma(Y))$. By Proposition 2.4, we have that $f(\bigvee F)=\bigvee f(F)$ holds for each irreducible subset  $F$ in $(P,\Upsilon(P))$ whenever $\bigvee F$ exists. By the universal property, there exists a unique continuous mapping $\overline{f}: \overline {P}\longrightarrow (Y, \sigma(Y))$ such that $\overline{f}\circ\eta=f$. The following proof is similar to  Case 2 of the proof of Theorem 3.3, we always obtain a contradiction.

 Thus {\bf KBSob} is not a reflective subcategory of {\bf Top$_k$}.
\end{proof}

\section{$QK$-bounded sober spaces}

Let $X$ be a $T_{0}$ space. Define a relation $\sim$ on $KB(X)$ as follows:
$$\mbox{For all\ } F,G\in KB(X), F\sim G  \mbox{ if and only if\ } \bigvee F=\bigvee G.$$
Then $\sim$ is an equivalence relation on $KB(X)$. Let $A\in KB(X)$, we write
$$[A]=\{B\in KB(X)\mid A\sim B\},$$
that is, $[A]$ is the equivalence class of $A$. Let $KB(X)/\!\sim$ denote the set of all equivalence class of $\sim$, that is, $KB(X)/\!\sim=\{[A]\mid A\in KB(X)\}$.
Define a mapping $q: KB(X)\longrightarrow KB(X)/\!\sim$ as follows:
$$\forall\ A\in KB(X),\ q(A)=[A].$$
Then $\tau_{X}=\{\mathcal{U}\subseteq KB(X)/\!\sim\mid q^{-1}(\mathcal{U})$ is open in $KB(X)\}$ is the quotient topology on $KB(X)/\!\sim$. We shall use
$KB(X)/\!\sim$ to denote the corresponding topological space.

\begin{proposition}  Let $X$ be a  $T_{0}$ space, $\mathcal{B}\subseteq KB(X)/\!\sim$. Then $\mathcal{B}$ is closed in the quotient space $KB(X)/\!\sim$ if and only if
there exists an $SI$-closed subset $B$ of $X$ such that $\mathcal{B}=\{[\dn x]\mid x\in B\}$.
\end{proposition}

\begin{proof}Necessity. Let $\mathcal{B}$ be a closed  subset in the quotient space $KB(X)/\!\sim$. Then $q^{-1}(\mathcal{B})$ is closed in $KB(X)$, and thus
there exists $B\in\Gamma(X)$ such that $q^{-1}(\mathcal{B})=K_{B}$. Let $F$ be an irreducible subset of $X$ such that $\bigvee F$ exists and $F\subseteq B$. Then
$cl(F)\in K_{B}$, and thus $q(cl(F))=[cl(F)]\in \mathcal{B}$.  Obviously, $cl(F)\sim\dn(\bigvee F)$.
It follows that $\dn(\bigvee F)\in q^{-1}(\mathcal{B})$.
Hence $\bigvee F\in B$, and so we conclude that $B$ is an $SI$-closed subset of $X$. It suffices to prove that $\mathcal{B}=\{[\dn x]\mid x\in B\}$.
Let $x\in B$. Then $\dn x\in K_{B}$, and thus $[\dn x]\in\mathcal{B}$. This shows that $\{[\dn x]\mid x\in B\}\subseteq \mathcal{B}$.
Conversely, let $[A]\in \mathcal{B}$. Then $[A]=[\dn(\bigvee A)]$, and thus $\dn(\bigvee A)\in q^{-1}(\mathcal{B})=K_{B}$.
We have that $\bigvee A\in B$. So $\mathcal{B}=\{[\dn x]\mid x\in B\}$.

\quad Sufficiency. It suffices to prove that $q^{-1}(\mathcal{B})=K_{B}$. Let $G\in K_{B}$. Then $G\subseteq B$.
 By Proposition 2.3, we have that  $\bigvee G\in B$. Then $[\dn(\bigvee G)]\in\mathcal{B}$, and thus $G\in q^{-1}(\mathcal{B})$.
 This shows $K_{B}\subseteq q^{-1}(\mathcal{B})$.
Conversely, let $G\in q^{-1}(\mathcal{B})$. Then $[\dn(\bigvee G)]\in \mathcal{B}$, and thus there exists $x\in B$ such that $[\dn(\bigvee G)]=[\dn x].$ It follows
that $x=\bigvee G\in B$. So $G\subseteq B$. This means  $q^{-1}(\mathcal{B})\subseteq K_{B}$. We have that $q^{-1}(\mathcal{B})=K_{B}$ is a closed subset in $KB(X)$. Hence
$\mathcal{B}$ is closed in the quotient space $KB(X)/\!\sim$.\end{proof}

\begin{theorem} Let $X$ be a  $T_{0}$ space. Then $KB(X)/\!\sim$ is  homeomorphic to  $SI(X)$.
\end{theorem}

\begin{proof}Define a mapping  $f:KB(X)/\!\sim\longrightarrow SI(X)$ as follows:
$$\forall\ [A]\in KB(X)/\!\sim,\ f([A])=\bigvee\!A.$$
Obviously, $f$ is a bijection.
Let $B$ be an $SI$-closed subset of $X$. By Proposition 4.1,  we have that $$f^{-1}(B)=\{[A]\mid \bigvee A\in B\}=\{[\dn x]\mid x\in B\}.$$
Then $f^{-1}(B)$ is a closed subset in the
quotient space $KB(X)/\!\sim$, and thus $f$ is a continuous mapping.
Let $\mathcal{B}$ be a closed subset in the quotient space $KB(X)/\!\sim$.  By Proposition
4.1, there exists an $SI$-closed subset $B$ of $X$ such that $\mathcal{B}=\{[\dn x]\mid x\in B\}$.
Then $f(\mathcal{B})=B$,
and thus $f$ is a closed mapping.
So $KB(X)/\!\sim$ is homeomorphic to $SI(X)$.\end{proof}

\begin{remark} Let $X$ be a $T_{0}$ space. Then $KB(X)$ is a k-bounded sober space. However, $KB(X)/\!\sim$ may not be a $k$-bounded sober space.
Please see the following Example.
\end{remark}

\begin{example}
Let $P$ be the poset as in Theorem 3.3. We consider the Alexandroff topological space $(P,\Upsilon(P))$.
As we know,  $(P,\sigma(P))$ is not a $k$-bounded sober space.
By Theorem 4.2,  $(P,\sigma(P))$ is homeomorphic to $KB(P)/\!\sim$. Then
$KB(P)/\!\sim$ is not a $k$-bounded sober space.
\end{example}

\begin{definition}
A $T_{0}$ space $X$ is said to be a $qk$-bounded sober space if the quotient space $KB(X)/\!\sim$ is a k-bounded sober space.
\end{definition}

\begin{example}
(1) Every k-bounded sober space is a $qk$-bounded sober space.

(2) Let $(P,\sigma(P))$ be the $T_{0}$ space as in Theorem 3.3. Then $(P,\sigma(P))$ is a $qk$-bounded sober space.

(3) Let $(P,\Upsilon(P))$ be the $T_{0}$ space as in Example 3.6. Then $(P,\Upsilon(P))$ is a $qk$-bounded sober space.
\end{example}

Let {\bf QKBSob} denote the category of all $qk$-bounded sober spaces and continuous mapping preserving existing irreducible suprema.
Then {\bf KBSob} is a full subcategory of {\bf QKBSob}.

\begin{theorem} {\bf KBSob}  is a full reflective subcategory of {\bf QKBSob}.
\end{theorem}

\begin{proof}Let $X$ be a $qk$-bounded sober space. Then $KB(X)/\!\sim$ is a $k$-bounded sober space. Define a mapping  $\eta: X\longrightarrow KB(X)/\!\sim$ as  follows:
$$\forall\ x\in X,\ \eta(x)=[\dn x].$$
By Proposition 4.1, we have that $\eta$ is a continuous mapping preserving  existing irreducible suprema. Let $Y$ be a $k$-bounded sober space and let $f: X\longrightarrow Y$
 be a continuous mapping preserving existing irreducible suprema. Define a mapping $\overline{f}: KB(X)/\!\sim\longrightarrow Y$ as follows:
 $$\forall\ A\in KB(X)/\!\sim,\ \overline{f}([A])=f(\bigvee A).$$
 Clearly, $\overline{f}$ is well defined. Let $B$ be a closed subset of  $Y$. By Proposition 2.6,  $B$ is an $SI$-closed subset of  $Y$.
 Then $f^{-1}(B)$  is an $SI$-closed subset of $X$.
 Since $$\overline{f}^{-1}(B)=\{[A]\mid f(\bigvee A)\in B\}=\{[\dn x]\mid x\in f^{-1}(B)\},$$
  $\overline{f}^{-1}(B)$ is a closed subset in the quotient space $KB(X)/\!\sim$. Then $\overline{f}$ is a continuous
 mapping. Obviously, $\overline{f}\circ\eta=f$.
 Let $g: KB(X)/\!\sim\longrightarrow Y$ be a continuous mapping such that $g\circ\eta=f$. For every $[A]\in KB(X)/\!\sim$,
 $$g([A])=g([\dn (\bigvee A)])=g(\eta(\bigvee A))=f(\bigvee A)=\overline{f}([A]).$$ Then $g=\overline{f}$.
 Thus {\bf KBSob}  is a full reflective subcategory of {\bf QKBSob}.\end{proof}

\bibliographystyle{amsplain}

\end{document}